\documentclass[a4paper,12pt,reqno]{mymathart}
\usepackage{amsmath, amsthm, amssymb, amscd, amsxtra}
\usepackage{mathrsfs}
\usepackage{graphicx} 
\usepackage{psfrag}

\usepackage[pdftex]{thumbpdf} 

\usepackage[latin1]{inputenc} 
\usepackage[T1]{fontenc}

\usepackage{float}

\newtheorem{Theorem}{Theorem}[section]
\newtheorem{Lemma}[Theorem]{Lemma}
\newtheorem{dfnandprop}[Theorem]{Definition and Proposition}
\newtheorem{Proposition}[Theorem]{Proposition}

\theoremstyle{remark}

\newtheorem*{proof of claim}{Proof of claim}

\theoremstyle{definition}
\newtheorem{Definition}[Theorem]{Definition}
\newtheorem*{Remark}{Remark}
\newtheorem*{exm}{Example}
\newtheorem*{dsa}{Dynamical standing assumption}
\newtheorem*{notations}{Notations}
\newtheorem*{asa}{Holomorphic standing assumption}

\def\F{{\mathcal{F}}}
\def\J{{\mathcal{J}}}
\def\A{{\mathcal{A}}}

\def\B{{\mathscr{B}}}
\def\C{{\mathbb{C}}}
\def\D{{\mathbb{D}}}
\def\N{{\mathbb{N}}}

\def\H{{\mathbb{H}}}
\def\Fl{{\mathscr{F}}}
\def\Fln{{\mathscr{F}}_{n}}
\def\Fli{{\mathscr{F}}_{\infty}}
\def\H{{\mathcal{H}}}

\newcommand{\e}{\operatorname{e}}

\newcommand{\id}{\operatorname{id}}

\newcommand{\Ima}{\operatorname{Im}}

\newcommand{\cl}{\overline}
\newcommand{\wt}{\widetilde}
\newcommand{\luc}{\operatorname{luc}}
\newcommand{\dyn}{\operatorname{dyn}}

\newcommand{\Hol}{\operatorname{Hol}}

\providecommand{\Attr}{\mathop{\rm Attr}\nolimits}

\usepackage{hyperref}
\title[Kernels of nonescaping-hyperbolic components]{Dynamical approximation and kernels of nonescaping-hyperbolic components}
\author{Helena Mihaljevi\'{c}-Brandt}

\begin{document}

\maketitle

\begin{abstract}
Let $\Fln$ be families of entire functions, 
holomorphically parametrized by a complex manifold $M$.
We consider those parameters in $M$ that correspond to 
\emph{nonescaping-hyperbolic} functions, i.e., those maps $f\in\Fln$
for which the postsingular set $P(f)$ is a compact subset of 
the Fatou set $\F(f)$ of $f$. 
We prove that if $\Fln\to\Fli$ in the sense of a certain dynamically 
sensible metric, 
then every nonescaping-hyperbolic component in  
the parameter space of $\Fli$ is a \emph{kernel} 
of a sequence of nonescaping-hyperbolic components in the 
parameter spaces of $\Fln$.
Parameters belonging to such a kernel do not 
always correspond to hyperbolic functions in $\Fli$. 
Nevertheless, we show that these functions must be $J$-stable. 
Using quasiconformal equivalences, we are able to construct many 
examples of families to which our results can be applied.  
\end{abstract}

\section{Introduction}

Let $\Hol^{*}(\C)$ be the space of all entire functions that are
not constant or linear. A map $f\in \Hol^{*}(\C)$ is said to
be \emph{nonescaping-hyperbolic} if its postsingular set $P(f)$ is a compact subset of 
the Fatou set $\F(f)$. 
(Basic definitions and notations are revised in 
Section \ref{section_pre_cons}.)
We want to understand how nonescaping-hyperbolic
maps behave under certain small perturbations.
It is known that within the space of all 
polynomials or the space of all 
transcendental entire maps with finitely many singular values, 
(nonescaping-)hyperbolic functions exhibit particularly simple and 
stable dynamics \cite{mane_sad_sullivan,eremenko_lyubich_2}.

The space $\Hol^{*}(\C)$ is naturally 
equipped with the topology of locally 
uniform convergence. However, this topology is not convenient
for our dynamical considerations since maps that are nearby in
the corresponding metric often have completely different dynamics 
(see e.g. Example \ref{ex1}). 
We introduce a new metric $\chi_{\dyn}$ on $\Hol^{*}(\C)$
which is dynamically more sensible: 
two entire maps are close with respect to $\chi_{\dyn}$  
if their locally uniform distance  
\emph{and} the Hausdorff distance between their sets of 
singular values is small. 

There are certainly many examples of entire maps that converge
in the metric $\chi_{\dyn}$ to a map in $\Hol^{*}(\C)$; 
the probably most prominent is the approximation of 
the exponential map
$E(z)=\e^z$ by the polynomials
$P_n(z)=\left(1 + z/n\right)^n$. 
In the article \cite{devaney_goldberg_hubbard}, Devaney, Goldberg and Hubbard  
investigated the relationship between the
families $E_{\lambda}(z)=\lambda\e^{z}$ 
and $P_{n,\lambda}(z)=\lambda\left(1 + z/n\right)^n$. 
The result was a nice connection between the respective parameter spaces: 
the authors proved pointwise convergence of 
nonescaping-hyperbolic components 
(i.e. connected components of the set of those $\lambda\in\C$ for which 
$E_{\lambda}$ or $P_{d,\lambda}$, respectively, is nonescaping-hyperbolic)
as well as convergence of certain external rays to curves called ``hairs'' 
in exponential parameter space (see also \cite{devaney_etal}). 
One can say that this point of view has provided an important 
conceptional basis for much subsequent work on the exponential family.

In this article, we embed the underlying approximation idea
into a general setup.
More precisely, let $M$ be a complex manifold and let $\chi_{M^{'}}$ be 
a metric on $M^{'}:=(\N\cup\{\infty\})\times M$. Let $\Hol^{*}_b(\C)$ denote the 
set of all maps in $\Hol^{*}(\C)$ with bounded sets of singular values. 
Our main result is the following.

\begin{Theorem}
\label{thm_main}
For every $n\in\N\cup\{\infty\}$ let 
$\Fln=\{ f_{n,\lambda}\}\subset\Hol^{*}_b(\C)$ 
be a family of entire functions 
that depend holomorphically on $\lambda\in M$. 
Assume that for every $n$, the singular values of all maps
in $\Fln$ are holomorphically parametrized by $M$, 
and that $F: M^{'}\to\Hol^{*}_b(\C), (n,\lambda)\mapsto f_{n,\lambda}$
is continuous with respect to the metrics 
$\chi_{M^{'}}$ and $\chi_{\dyn}$.
 
If $\wt{H}$ is a kernel 
of a sequence of hyperbolic components of $\Fln$, then exactly one of the following 
statements holds: 
\begin{itemize}
\item[$(i)$] The map $f_{\lambda}\in\Fli$ is not hyperbolic for any $\lambda\in\wt{H}$.
\item[$(ii)$] There is a hyperbolic component $H_{\infty}$ of $\Fli$ such that 
$\wt{H}=H_{\infty}$. 
\end{itemize}
\end{Theorem}

\begin{Remark}
Our result is a natural generalization of a theorem by 
Krauskopf and Kriete, who  
considered holomorphic families 
$\Fln=\{f_{\lambda}: \lambda\in\C\}$, $n\in\N\cup\{\infty\}$, of entire maps 
for which the sets of singular values are holomorphically parametrized
and for which there exists an integer $N$ with 
$\vert S(f_{n,\lambda})\vert = N<\infty$ for all $n\in\N\cup\{\infty\}$ 
and all $\lambda\in\C$.
They proved the same conclusions as in Theorem \ref{thm_main}, 
provided that $\Fln\to\Fli$ uniformly
on compact subsets of $\C\times\C$ \cite{krauskopf_kriete}. 
Our proof of the above Theorem follows the same idea as in 
\cite{krauskopf_kriete}.
\end{Remark}

The first case in Theorem \ref{thm_main} 
does indeed occur; 
an example is given in Section \ref{section_main}. Nevertheless, 
parameters that belong to a kernel 
define maps which exhibit certain stability.
Under the same assumptions (and notations) as in Theorem \ref{thm_main} 
we prove the following result.
\begin{Theorem}
\label{thm_main2}
Let $\lambda$ belong to a kernel $\wt{H}$. Then 
$f_{\lambda}\in\Fli$ is a $J$-stable function. 
\end{Theorem}
We define $J$-stability analogously to the classical definition for 
rational maps 
(see Definition \ref{dfn_Jstable}). 

Given an entire function $f$ together with a sequence of entire maps 
$f_n$ for which $\chi_{\dyn}(f_n,f)\to 0$ when $n\to\infty$,
there is a natural way -- using quasiconformal equivalence classes --
to construct suitable holomorphic families $\Fln\ni f_n$ that
satisfy the assumptions of Theorem \ref{thm_main}. 
For a lot of explicit functions, including many examples
that have been of particular dynamical interest in the last decades, 
non-trivial approximations in the sense of $\chi_{\dyn}$ are known.
However, for an \emph{arbitrary} (transcendental) entire map 
there seems to be no 
general concept of how to establish a non-trivial sequence of maps $f_n$ 
such that $\chi_{\dyn}(f_n,f)\to 0$. This 
leaves open an interesting function-theoretic question.

\subsection*{Structure of the article}
In Section \ref{section_pre_cons} we introduce 
basic definitions, notations and preliminary concepts.
Section \ref{section_main} addresses the proofs of Theorems 
\ref{thm_main} and \ref{thm_main2}. In the final part, we discuss examples
to which our results can be applied.

\subsection{Acknowledgements}
I want to thank Lasse Rempe for his encouragement to extend the results
from my Diploma thesis and for his 
continual help and support.
I would also like to thank Hartje Kriete for introducing me
to the subject during my time as a graduate student in G\"{o}ttingen.

\section{Preliminary constructions}
\label{section_pre_cons}
We denote the complex plane by $\C$ and the 
Riemann sphere by $\widehat{\C}:=\C\cup\lbrace\infty\rbrace$. 
For a subset $V$ of a metric space $M$ we denote the 
$\varepsilon$-neighbourhood of $V$ by $U_{\varepsilon}(V)$. 
We will write $A\Subset B$ 
if $A$ is a relatively compact subset of $B$, 
i.e., if $\cl{A}$ is compact and contained in $B$.
If not stated differently, 
the boundary $\partial A$ and the closure $\overline{A}$ 
of a set $A\subset\C$ is always understood to be taken 
relative to the complex plane.
Throughout this article, we will mainly consider
nonconstant, nonlinear entire functions  $f:\C\rightarrow\C$,
hence $f$ will either be a polynomial of degree $\geq 2$
or a transcendental entire map. Recall that the space of all
such maps is denoted by $\Hol^{*}(\C)$.

We denote by $C(f):=\{w\in\C:\exists z: f^{'}(z)=0\text{ and } f(z)=w\}$ 
the set of all critical values
and by $A(f)$ the set of all (finite) asymptotic values of $f$. 
Recall that a point $w\in\C$ is an asymptotic
value of $f$ if there is a path to $\infty$ along which
$f$ converges to $w$. 
Note that a polynomial has no asymptotic values.
The set $S(f)$ of \emph{singular values} of $f$  
is defined to be the smallest closed set in $\C$ such that  
$f: \C\setminus f^{-1}(S(f))\rightarrow \C\setminus S(f)$
is a covering map. 
It is well-known that $S(f)$ can be written as 
$S(f)=\overline{C(f)\cup A(f)}$ (see e.g. \cite[Lemma 1.1]{goldberg_keen}).
We define 
\begin{align*}
\Hol^{*}_b(\C):=\{ f\in\Hol^{*}(\C): S(f)\text{ is bounded}\}
\end{align*}
to be the set of all nonconstant, nonlinear entire maps with bounded 
singular sets. 
Finally, we denote the \emph{postsingular set} of $f$ by 
$P(f):=\overline{\bigcup_{n\geq 0} f^{n}(S(f))}$.

Let $w\in\C$ be a periodic point of $f$ of period $k$, that is, $f^k(w)=w$ 
and $k>0$ is minimal with this property.
The set $O(w):=\lbrace w, f(w),...,f^{k-1}(w)\rbrace$ 
is called the (forward) \emph{orbit} or \emph{cycle} of $w$. 
The \emph{multiplier} $\mu(w)$ of $w$ is defined as 
$\mu(w):= (f^k)^{'}(w)$. 
We call $w$ \emph{attracting} if $\vert\mu(w)\vert < 1$, 
\emph{repelling} if $\vert\mu(w)\vert > 1$ 
and \emph{indifferent} if $\vert\mu(w)\vert = 1$.
An attracting periodic point $w$ with 
$\mu(w)=0$ is called \emph{superattracting}.
If $w$ is a periodic point then $\mu(f^i(w))=\mu(w)$ for all $i$, 
hence we can extend the classification of $w$ to the whole cycle $O(w)$
and speak of the multiplier of the cycle etc.

If $O$ is an attracting periodic cycle of $f$ of period $n$, then
the \emph{basin of attraction} or \emph{attracting basin of $O$} 
is the open set $A(O)$  consisting of all points $z$
for which the successive iterates $f^{n}(z), f^{2n}(z),\dots$ converge 
to some point of $O$. If $w\in O$ then the component $A^{*}(w)$ of
$A(O)$ that contains $w$ is called the \emph{immediate basin of attraction of $w$}.
(The immediate basin of the cycle $O$ is then the union of 
the immediate basins of the points in $O$.)
We denote the set of all points that converge to an attracting cycle 
of $f$ by $\A(f)$.

Let $f\in\Hol^{*}(\C)$.
The objects of main dynamical interest related to $f$ are 
the \emph{Fatou set} $\F(f)$ of $f$, 
defined as the set of all points 
that have a neighbourhood in which the iterates
$(f^n)_{n\in\N}$
form a normal family in the sense of Montel,
and its complement the \emph{Julia set} $\J(f):= \C\setminus\F(f)$ of $f$.

\begin{Remark}
A polynomial $f$ extends naturally to a map 
$\widehat{f}:\widehat{\C}\to\widehat{\C}$ with $\widehat{f}(\infty)=\infty$.
In this case, $\infty$ is a critical point and a critical value of $\widehat{f}$.
However, the dynamically relevant sets are
deliberately chosen to be 
subsets of the plane (rather than the sphere) 
since the consideration of the point at $\infty$ 
does not contribute to the
understanding of the dynamics: for every transcendental entire map, 
$\infty$ is an essential singularity, while for a polynomial, $\infty$ is its only
preimage.
\end{Remark}

The \emph{escaping set} of $f$ is defined to be
\begin{align*}
I(f):=\{z\in\C: \lim_{n\to\infty} f^n(z) = \infty\}.
\end{align*}
A point $z\in I(f)$ is called an \emph{escaping point}. 
For a polynomial $f$ (of degree $\geq 2$), the escaping set belongs to $\F(f)$,
since the point at $\infty$ can be considered as a superattracting fixed 
point of the holomorphic extension of $f$ to $\widehat{\C}$ 
and the escaping set as its basin of attraction. In contrast, 
if $f\in\Hol^{*}_b(\C)$ is transcendental entire, then the relation
$\J(f)=\overline{I(f)}$ holds \cite[p. 344]{eremenko_1}.

For more background in holomorphic dynamics, we refer
to the expository works \cite{bergweiler_1,milnor}.
The classification of components of the Fatou set of
an entire map will be used implicitly throughout this article  and
can be found in \cite{bergweiler_1}.

\subsection*{Nonescaping-hyperbolic maps}

Let us start with the definitions of hyperbolic and 
nonescaping-hyperbolic functions.

\begin{Definition}
\label{def_hyp}
A map $f\in\Hol^{*}(\C)$ is called \emph{hyperbolic} 
if $S(f)$ is bounded and $P(f)\subset\F(f)$.
We say that $f$ is  
\emph{nonescaping-hyperbolic} if $P(f)$ is a bounded 
subset of $\F(f)$. 
\end{Definition}

Clearly, every nonescaping-hyperbolic map is also hyperbolic.
Recall that the set $S(f)$ is closed by definition, hence if 
$f$ is hyperbolic then 
$S(f)$ is a compact subset of $\F(f)$.
Similarly for $P(f)$ when $f$ is nonescaping-hyperbolic. 

The next observation will be used frequently;
it follows mainly from classical results in holomorphic dynamics 
but we could not localize a reference, 
hence we include a proof for completeness.

\begin{Proposition}
\label{prop_hyp}
Let $f\in\Hol^{*}(\C)$. 
Then $f$ is nonescaping-hyperbolic if and only if 
$S(f)$ is a bounded subset of $\A(f)$.

Let $f$ be nonescaping-hyperbolic. Then $\Attr(f)$ is finite, 
and if $f$ is transcendental entire then
$\F(f)=\A(f)\neq\emptyset$, otherwise $\F(f)=\A(f)\cup I(f)$.
\end{Proposition}

\begin{proof}
The statement is well-known for polynomials \cite[Theorem 19.1]{milnor} 
hence we can restrict to 
transcendental entire maps.

Assume that $S(f)$ is bounded and every $w\in S(f)$ converges to an attracting
cycle of $f$. Then the components of $\A(f)$ form an open cover 
of $S(f)$ and since $S(f)$ is compact, there is a finite subcover.
In particular, it follows that $P(f)$ is a compact subset of $\A(f)\subset\F(f)$,
hence $f$ is nonescaping-hyperbolic. 

Now assume that 
$P(f)$ is a compact subset of $\F(f)$ and 
let $U$ be a component of $\F(f)$. 
Then the iterates $f^n(z)$ do not converge to $\infty$ for any $z\in U$ 
\cite[Theorem 1]{eremenko_lyubich_2}, 
hence $U$ is not a Baker domain.
This means that $U$ is either a component of an attracting or 
parabolic basin, a Siegel disk or 
a wandering domain \cite[Theorem 6]{bergweiler_1}. 
Assume that $U$ is a wandering domain and let $z\in U$. Then $f^n(z)$ 
accumulates at points in $\J(f)\cap P(f)$ \cite[Theorem]{bergweiler_etal} 
but this set is empty. 
Hence $U$ is not a wandering domain. 
For the same reason, 
$U$ is not a Siegel disk or a component of a 
parabolic basin (see \cite[Theorem 7]{bergweiler_1}),
hence $U$ is a component of a basin of attraction of an
attracting periodic orbit. Hence 
$\F(f)=\A(f)$ and 
every singular value converges under iteration to 
an attracting periodic orbit.

Now let $f$ be nonescaping-hyperbolic. We have shown that
only finitely many attracting cycles contain a singular value in their 
basins, and since every attracting cycle must intersect $S(f)$ 
\cite[Theorem 7]{bergweiler_1}, it follows that $\Attr(f)$ is finite.
\end{proof}

Note that the proof of Proposition \ref{prop_hyp} implies that
for a transcendental entire map, hyperbolicity and 
nonescaping-hyperbolicity are equivalent. 
Furthermore, a polynomial $f$ is hyperbolic but not
nonescaping-hyperbolic if and only if $S(f)\cap I(f)\neq\emptyset$.

By Definition \ref{def_hyp}, every (nonescaping-)hyperbolic 
transcendental entire function belongs to the \emph{Eremenko-Lyubich class} 
\begin{eqnarray*}
\B:=\lbrace f:\C\rightarrow\C \text{ transcendental entire}: S(f) \text{ is bounded}\rbrace. 
\end{eqnarray*}

Finally, we would like to remark 
that an entire map for which every singular value is absorbed
by an attracting cycle does not necessarily have dynamics similar to
a hyperbolic function.
The requirement 
that $S(f)$ is a bounded set is indeed crucial for Proposition \ref{prop_hyp} to hold:
As Example D in \cite{kisaka_shishikura} shows, 
there is a transcendental entire map $f$ with wandering domains for 
which all singular values are mapped by $f$ to an attracting fixed point.

\subsection*{Hausdorff and kernel convergence}
\label{subs_con_sets}
Let $M$ be a metric space. 
Recall that the \emph{Hausdorff distance} between two compact 
sets $A,B\subset M$ is defined to be
\begin{eqnarray*}
d_H(A,B):=\inf \lbrace \varepsilon>0: A\subset U_{\varepsilon}(B),\;B\subset U_{\varepsilon}(A) \rbrace.
\end{eqnarray*}

For studying convergence of open connected subsets of $M$ 
we will define a concept analogous to the \emph{Carath\'eodory} or 
\emph{kernel convergence} for domains in the complex plane. 
For more details, see \cite[$\S 5$]{golusin}.

\begin{Definition}[Kernel]
\label{def_kernel}
Let $o\in M$ and let $O_n\subset M$, $n\in\N$, 
be open connected sets containing $o$. 
The \emph{kernel} of the sequence $(O_n)$ (w.r.t. $o$) 
is the largest open connected set $O\ni o$ 
such that each compact set $K\subset O$ is contained in 
all but finitely many $O_n$.
\end{Definition}

We call the point $o$ the \emph{marked point} of the sequence $(O_n)$.
Clearly, the existence of a kernel is equivalent to the existence 
of a neighbourhood of $o$ which is contained in all but finitely many $O_n$. 

We say that the sequence $(O_n)$ \emph{converges} to $O$ (\emph{as kernels}) 
and write $O_n\rightarrow O$ if $O$ is 
a kernel of each subsequence of $(O_n)$. 
The middle example in 
Figure \ref{figure_1} shows a sequence that does not converge
to its kernel. 
 
Observe that a sequence $(O_n)$ can have more than one kernel 
(see the left-hand example in Figure \ref{figure_1}), 
and each of them is specified by the choice of a marked point. 
Now let $M$ be a locally compact metric space, e.g., an analytic manifold,
and let $O$ be a kernel of $O_n$. 
Since, by definition, every kernel is open and since $M$ is locally compact,
every point in $O$ has a compact neighbourhood which is contained in $O$. 
This means that we can choose any point $o\in O$ 
to be the marked point. 
Hence we will talk about the sets $O_n$ and $O$ 
without mentioning the marked point if it is implicit from 
the context which point is meant.

We have the following relation between Hausdorff and kernel convergence.
\begin{Proposition}
\label{prop_kernel_hausdorff}
Let $K_n, K$ be nonempty compact subsets of a locally compact metric space $M$. 
Then $d_H(K_n,K)\rightarrow 0$ as $n\rightarrow\infty$ 
if and only if the following two conditions hold:
\begin{itemize}
\item every component $O$ of $K^c:=M\backslash K$ is a 
kernel of a sequence of components $O_n$ of $K_n ^c:=M\backslash K_n$,
\item every kernel of an infinite sequence $(O_{n_k})$ of 
components of $K_{n_k}^c$ is a component of $K^c$.
\end{itemize}
\end{Proposition}
We will omit the proof since it is elementary and 
follows mainly from the definitions of kernel and Hausdorff convergence,
and since we do not require the statement for any proof 
in this article; its purpose is more the illustration of 
how the given concepts of convergence relate to each other.
The right-hand example in Figure \ref{figure_1} shows the 
necessity of the second requirement in Proposition \ref{prop_kernel_hausdorff}.

\begin{figure}
\psfrag{K1}{$K_1$}
\psfrag{K2}{$K_2$}
\psfrag{K3}{$K_3$}
\psfrag{D2n}{$D_{2n}$}
\psfrag{D2n+1}{$D_{2n+1}$}
\psfrag{K}{$\mathbb{H}_{< 0}$}
\psfrag{i/n}{$\frac{i}{n}$}
\psfrag{-i/n}{$\frac{-i}{n}$}
\psfrag{null}{$0$}
\psfrag{n}{$n$}
\centering
\begin{minipage}{.32\linewidth}
\raggedleft
\includegraphics[width=\linewidth]{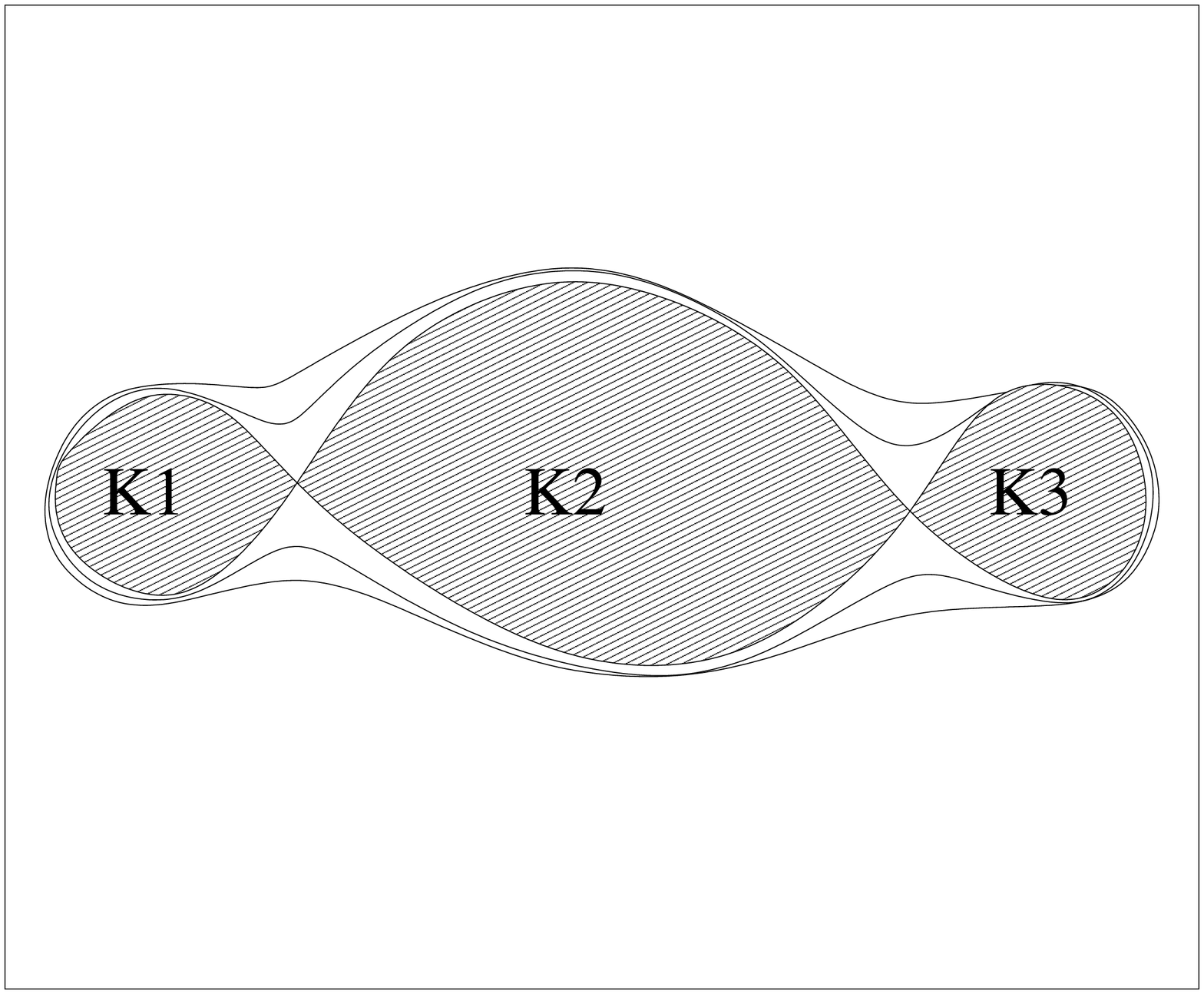}

\end{minipage}
\hspace{.002\linewidth}
\begin{minipage}{.32\linewidth}
\centering
\includegraphics[width=\linewidth]{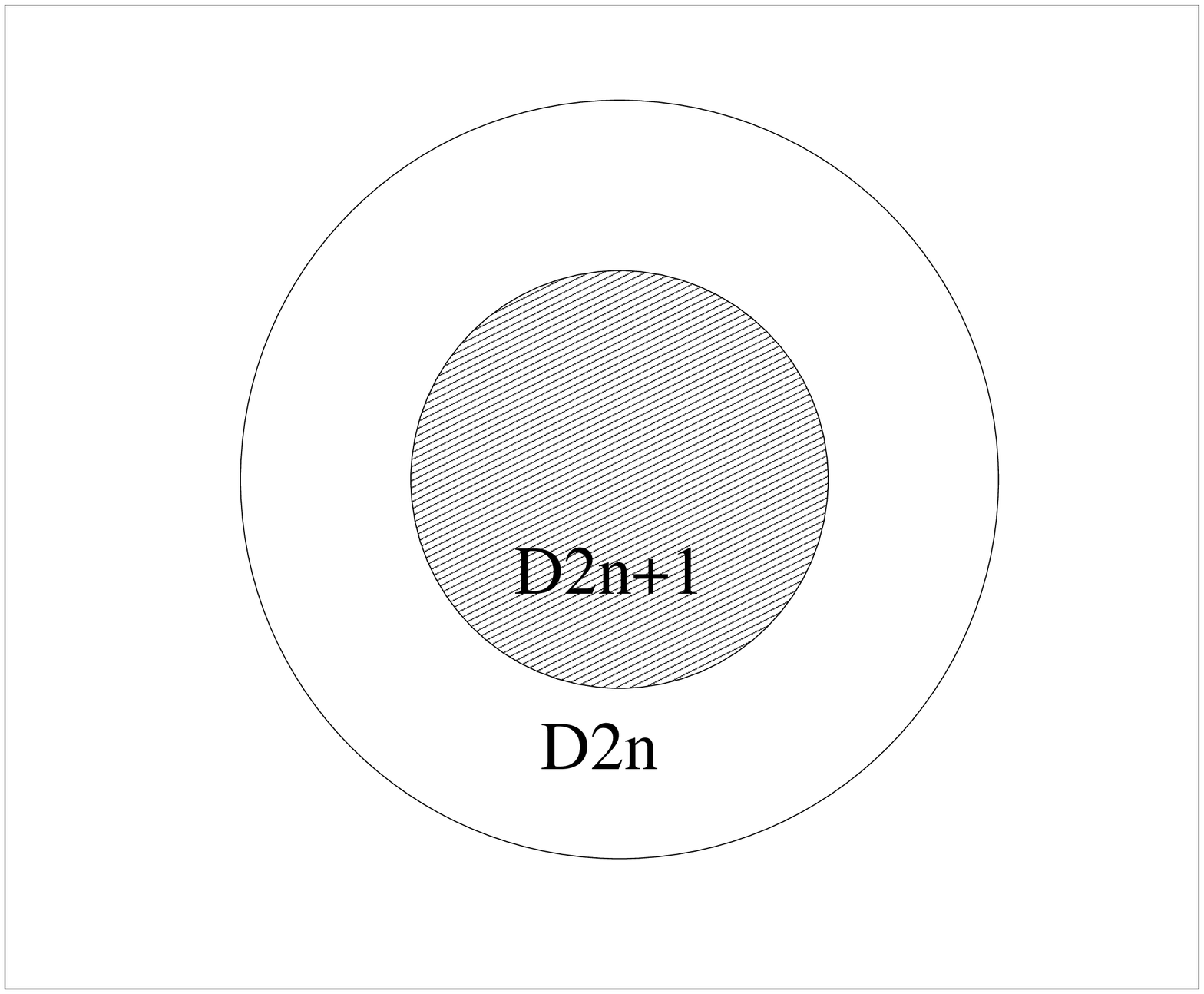}

\end{minipage}
\hspace{.002\linewidth}
\begin{minipage}{.32\linewidth}
\raggedright
 \includegraphics[width=\linewidth]{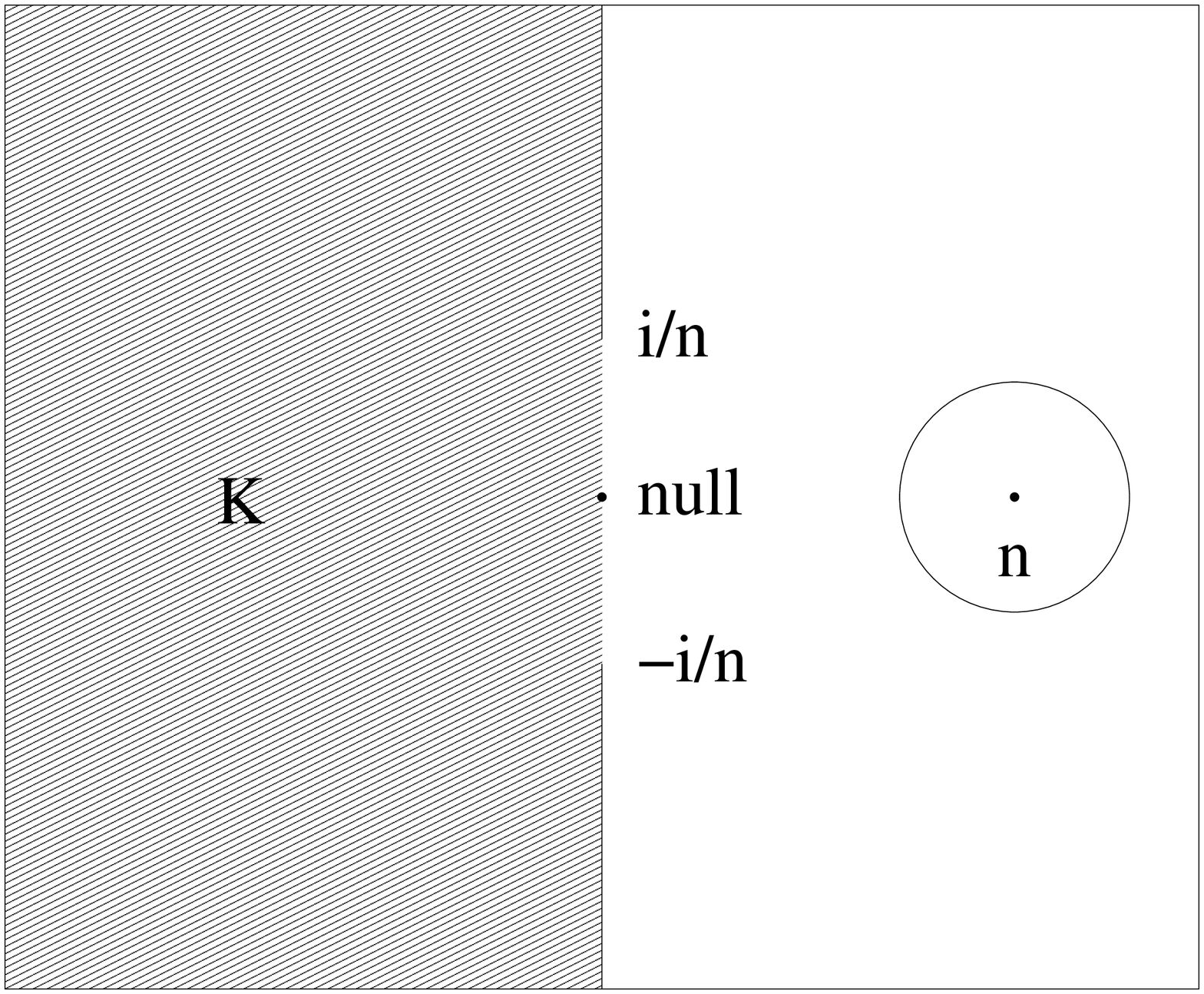}

\end{minipage}
\caption{\emph{left:} The interiors of the curves are simply-connected 
domains with three different kernels $K_1$, $K_2$ and $K_3$. 
\emph{middle:} The domains 
$D_n:=\{z: \vert z\vert <2-r_n\text{ where } n\equiv r_n\mod 2, r_n\in\{0,1\}\}$ 
have a unique kernel $D_1$  
but they do not converge to it. 
\emph{right:} The domains $D_n:=\C\backslash \left( \lbrace z= i y: \vert y\vert\geq 1/n\rbrace\cup\lbrace z: \vert z-n\vert \leq 1\rbrace\right)$ 
converge as kernels 
to the left half-plane $\mathbb{H}_{< 0}$
w.r.t. the marked point $-1$ 
but their complements do not converge to 
$\C\setminus \mathbb{H}_{< 0}$ in the Hausdorff metric.}
\label{figure_1}
\end{figure}

\subsection*{Dynamical approximation}
\label{subs_dyn_appr}

Let $\Hol(\C)$ be the space of all entire functions.
We denote the \emph{locally uniform distance} 
between $f,g\in \Hol(\C)$ by $\chi_{\luc}(f,g)$. 
The metric $\chi_{\luc}(f,g)$ 
induces the topology of locally uniform convergence on $\Hol(\C)$, 
so we say that $f_n$ converge to $f$ \emph{locally uniformly} if and only if $\chi_{\luc}(f_n,f)\rightarrow 0$ as $n\rightarrow\infty$. 
(For more details see e.g. \cite[Chapter 3]{milnor}.)
It follows from the Weierstra{\ss} Approximation Theorem 
\cite[Theorem 1.4]{milnor} that 
the space $\Hol(\C)$ is closed 
with respect to this topology.

For entire maps with a 
non-empty set of singular values we introduce a new metric 
which combines locally uniform convergence with controlled behaviour 
on the set of singular values. 
Hence this metric will be more convenient for the study of dynamics of entire functions.
Recall that $\Hol^{*}(\C)$ denotes the set of 
all entire functions which are not constant or linear.

\begin{dfnandprop}
\label{dyn_approx_def}
 The map $\chi_{\dyn}:\Hol^{*}(\C)\rightarrow [0,\infty)$ with 
\begin{eqnarray*}
 \chi_{\dyn}(f,g):=\chi_{\luc}(f,g) + d_H (S(f),S(g))
\end{eqnarray*}
is a metric, where $d_H (S(f),S(g))$ is measured with respect to the spherical metric. 

We will say that the sequence $f_n$ approximates $f$ \emph{dynamically} if $\chi_{\dyn}(f_n,f)\rightarrow 0$ as $n\rightarrow\infty$.
\end{dfnandprop}

Maps which are close in the metric $\chi_{\luc}$ do 
not necessarily have the property that their sets of 
singular values are close in the Hausdorff metric. 
This does not have to be true even in the case of a family 
of functions which depends holomorphically on some parameter $\lambda$, 
as the following example shows. 

\begin{exm}
\label{ex1}
Let $f_{\lambda}(z)=\e^{-\lambda z^2 + z-2}$ with $\lambda\in\C$. 
The map $f_0(z) = \e^{z-2}$ has $0$ as its only singular value. 
Now let $\lambda\neq 0$ be any complex number. 
Then apart from the asymptotic value $0$, the map 
$f_{\lambda}$ has the additional critical value $v_{\lambda}:=\e^{1/(4\lambda) -2}$. 
Clearly, $v_{\lambda}\to +\infty$ when $\lambda\searrow 0$.
\end{exm}

It is well-known that if $f_n$ is a sequence of entire maps 
such that $\chi_{\luc}(f_n,f)\rightarrow 0$, 
then for every $w\in S(f)$ there is a sequence 
$\{w_n: w_n\in S(f_n)\text{ for every }n\}$ 
such that $w_n\rightarrow w$ 
(see e.g. \cite[Theorem 2]{kisaka}), yielding 
lower semi-continuity for the sets of singular values 
in case of locally uniform convergence. 
Hence convergence in the metric $\chi_{\dyn}$ 
makes in particular sure that there are no sequences of singular values of the 
approximating maps $f_n$ that accumulate outside $S(f)$. 

Nonescaping-hyperbolicity is not an 
open property in the topology of locally uniform convergence. 
For instance, let $f_{\lambda}(z):=\e^{\lambda z}\cdot z^2$. Then
$f_0(z)=z^2$ is nonescaping-hyperbolic while for any sufficiently small 
$\lambda>0$, the critical value $4/(\e^2\lambda^2)$ escapes to 
$\infty$. 
However, the set of nonescaping-hyperbolic entire maps 
is open in the topology induced by the metric $\chi_{\dyn}$.

\begin{Theorem}[Nonescaping-hyperbolicity is an open property]
\label{hyp_open}
The set 
\begin{eqnarray*}
 \H:=\lbrace f\in \Hol^{*}(\C): f\text{ is nonescaping-hyperbolic}\rbrace
\end{eqnarray*}
is open in the topology induced by the metric $\chi_{\dyn}$.
\end{Theorem}

Note that $\H\subset\Hol^{*}_b(\C)$.
Theorem \ref{hyp_open} will follow from the following lemma.

\begin{Lemma}
\label{K_lemma}
Let $f\in\Hol^{*}(\C)$ and let $K\subset\A(f)$ be a compact set. 
Then $K\subset\A(g)$ for all $g\in\Hol^{*}(\C)$ 
that are sufficiently close to $f$ in the metric $\chi_{\luc}$.
\end{Lemma}

\begin{proof}
The components of $\A(f)$ form an open cover of the 
compact set $K$, hence there is a finite subcover. 
We can assume w.l.o.g. that $K$ is contained in the 
basin of attraction $A(z_0)$ of a fixed point $z_0\in\C$ of $f$, 
since otherwise we can repeat the argument for 
every attracting periodic point of $f$. 
There exists a bounded open set $U\ni z_0$ such that $f(U)\Subset U$. 
By definition of $\chi_{\luc}$, $g(U)\Subset U$ 
holds for all entire maps $g$ for which
$\chi_{\luc}(f,g)$ is sufficiently small. 
By Montel's Theorem, $\lbrace g^k\rbrace_{k\in\N}$ is normal in $U$ 
and by the Contraction Mapping Theorem, $g$ has a fixed point in $U$ 
which is necessarily attracting. 
Since $U$ is bounded, it follows that $U\subset\A(g)$. 

There exists $N\in\N$ such that $f^N (K)\subset U$, so again, 
by locally uniform convergence, $g^N (K)\subset U\subset\A(g)$ 
if $\chi_{\luc}(f,g)$ sufficiently small. 
The claim now follows from the complete invariance of $\A(g)$ 
under the map $g$.
\end{proof}

\begin{proof}[Proof of Theorem \ref{hyp_open}]
Let $f$ be nonescaping-hyperbolic, 
hence $S(f)$ is bounded and contained in $\A(f)$. 
Choose $\delta>0$ sufficiently small such that 
$K:=\overline{U_{\delta}(S(f))}\subset\A(f)$. 
Since $K$ is compact, it follows from Lemma \ref{K_lemma} 
that there is a constant $\varepsilon>0$ 
such that $K\subset\A(g)$ for all $g$ with $\chi_{\luc}(f,g)<2\varepsilon$. 
Now choose $\eta=\min\lbrace\varepsilon,\delta\rbrace$. 
Then for all maps $g$ with $\chi_{\dyn}(f,g)<\eta$ we obtain
\begin{eqnarray*}
S(g)\subset U_{\delta}(S(f))\Subset\A(g),
\end{eqnarray*}
hence $g$ is nonescaping-hyperbolic. 
\end{proof}

\begin{Remark}
We have shown that 
nonescaping-hyperbolicity is not an open property in the topology
of locally uniform convergence. 
The given example also shows that --- in the same topology --- 
the set of all hyperbolic maps in $\Hol^{*}(\C)$ is not open either.
Now let $p$ be a hyperbolic polynomial for which a finite singular
value escapes to infinity.
 It is plausible that for 
any sufficiently small $\varepsilon$, the neighbourhood 
$U_{\varepsilon}(p)=\{f\in\Hol^{*}(\C):\chi_{\dyn}(f,p)<\varepsilon\}$ of $p$  
contains transcendental entire functions with
the same property, i.e., maps for which some singular value escapes.
Hence it is likely that hyperbolicity is not open in the topology
induced by the metric $\chi_{\dyn}$ either.
\end{Remark}

\section{Stability of nonescaping-hyperbolic parameters}
\label{section_main}
Recall that our goal is to prove that under certain conditions,
a kernel of a sequence of nonescaping-hyperbolic components equals a 
nonescaping-hyperbolic component of the limit family.
As we will see, it is not hard to show that 
every such component of 
$\Fli$ is contained in a kernel of a sequence of 
nonescaping-hyperbolic components of $\Fln$. 
This statement requires even less restrictions than 
those stated in Theorem \ref{thm_main}.
For the other inclusion to hold, 
we have to construct a more sensible setup.

Denote by $\widehat{\N}:=\N\cup\lbrace\infty\rbrace$ the one-point 
compactification of the set of natural numbers.
The metric $\chi_{\widehat{\N}}$ on $\widehat{\N}$ defined by 
$\chi_{\widehat{\N}}(m,n):= 2^{-\min (m, n) -1}$ when $m\neq n$ 
(and equal zero otherwise)
makes $\widehat{\N}$ a complete metric space. 
From now on, we assume that $M$ 
is a complex manifold with a metric $\chi_M$ and 
define $M^{'}:=\widehat{\N}\times M$. 
The relation $\chi_{M^{'}} ((m,\lambda),(n,\nu))
:=\chi_{\widehat{\N}}(m,n) + \chi_{M^{'}} (\lambda,\nu)$ then 
defines a metric $\chi_{M^{'}}$ on $M^{'}$. 

For every $n\in\widehat{\N}$ let 
$\Fln=\lbrace f_{n,\lambda}:\C\rightarrow\C,\;
\lambda\in M\rbrace\subset\Hol^{*}_b(\C)$ 
be a family of functions parametrized by $M$.
To simplify the notations, 
we will skip the first index entry for maps 
in $\Fli$, i.e., we will write $f_{\lambda}$ 
for $f_{\infty,\lambda}\in\Fli$.

We want all families $\Fln, n\in\widehat{\N}$, to satisfy the following. 
\begin{dsa}['dsa']
The map 
\begin{eqnarray*}
 F:M^{'}\rightarrow\Hol^{*}_b(\C),\; (n,\lambda)\mapsto f_{n,\lambda}
\end{eqnarray*}

is continuous with respect to the metrics $\chi_{M^{'}}$ and $\chi_{\dyn}$. 
\end{dsa}

The key-feature of 'dsa' is the local uniformity in $\lambda$ \emph{and} $n$: 
Let $f_{\lambda_0}\in\Fli$ and let $K\subset\C$ be a compact set.
Then for every $\varepsilon>0$ there exist $\delta>0$ and $n_0\in\N$
such that $\vert f_{n,\lambda}(z)-f_{\lambda_0}(z)\vert<\varepsilon$
for all $z\in K$, $\lambda\in U_{\delta}(\lambda_0)$ and $n\geq n_0$.

\begin{notations}
For every $n\in\widehat{\N}$ we denote by
\begin{align*}
 \H(\Fln):=\{ \lambda\in M: f_{n,\lambda}\in\Fln \text{ is nonescaping-hyperbolic}\}
\end{align*}
the parameters corresponding to nonescaping-hyperbolic
maps in the respective family.
We will usually denote a component of $\H(\Fln)$ by $H_n$, 
a component of $\H(\Fli)$ by $H_{\infty}$ 
and a kernel of a sequence $H_n$ by $\widetilde{H}$.
\end{notations}

\begin{Proposition}
Let $H_{\infty}$ be a component of $\H(\Fli)$. 
Then there exists a kernel $\widetilde{H}$ of a sequence of
components of $\H(\Fln)$ such that $H_{\infty}\subset\widetilde{H}$.
\end{Proposition}

\begin{proof}
Let $\lambda_0\in\H(\Fli)$. 
It follows from Theorem \ref{hyp_open} and the dynamical standing assumption 
that there exists a neighbourhood $U(\lambda_0)\subset M$ 
such that $U(\lambda_0)\subset\H(\Fln)$ for all sufficiently large $n\in\N$.
\end{proof}

To prove the opposite inclusion, we have to make additional restrictions.
Our requirements, that we will assume from now on, 
are formalized in the following way.
\begin{asa}['hsa'] 
$\;$
\begin{itemize}
\item[$(i)$] For every $n\in\widehat{\N}$, 
the maps $f_{n,\lambda}$ depend holomorphically on $\lambda\in M$.
\item[$(ii)$] Let  $n\in\N$ and $\lambda_0\in M$. 
Then the singular values of $f_{n,\lambda_0}$ are 
holomorphically parametrized by $M$, 
i.e., for each singular value $s$ of $f_{n,\lambda_0}$ 
there exists a holomorphic map $w_n:M\rightarrow\C, \lambda\mapsto w_n(\lambda)$ 
such that $w_n(\lambda_0)=s$ and $w_n(\lambda)\in S(f_{n,\lambda})$.
\end{itemize}
\end{asa}
Note that the second condition of 'hsa' does not imply that 
a parametrization of $S(f_{n,\lambda_0})$ for some $\lambda_0$ needs to be 
an exhaustion of the set of singular values for another parameter 
$\lambda\neq\lambda_0$. A priori, it is possible that 
$S(f_{n,\lambda_0})=\lbrace w_n^{i}(\lambda_0)\rbrace_{i\in I}$ 
for some index set $I$ but 
$\lbrace w_n^{i}(\lambda)\rbrace_{i\in I}\subsetneq S(f_{n,\lambda})$ 
for some $\lambda\neq\lambda_0$. 

Also note that we do not assume condition $(ii)$ 
to be satisfied by the maps in $\Fli$. 
The reason is that we only need a \emph{local} 
holomorphic parametrization of the sets of singular values 
for maps in $\Fli$ and as the next statement shows, this
follows from the assumptions we already made.

\begin{Theorem}
\label{normality}
Let $\widetilde{H}$ be a kernel of a sequence of 
nonescaping-hyperbolic components $H_n$, $U\ni\lambda_0$ a simply-connected 
neighbourhood of $\lambda_0$ with $U\Subset\widetilde{H}$, and 
let $s$ be a singular value of $f_{\lambda_0}$.

Then there exists a holomorphic map 
$w:U\rightarrow\C$ such that 
$w(\lambda)\in S(f_{\lambda})$, $w(\lambda_0)=s$ and 
the family $\lbrace f_{\lambda}^n (w(\lambda))\rbrace_{n\in\N}$ 
is normal in $U$.
\end{Theorem}

\begin{proof}
Let $f_{\lambda_0}\in\Fli$. By \cite[Theorem 5]{bergweiler_1} 
and \cite[Theorem 13.1]{milnor}, the map $f_{\lambda_0}$ has infinitely
many repelling periodic points (in its Julia set). Let us pick
two such points and denote them by $p(\lambda_0)$ and $q(\lambda_0)$;
let $n_1$ and $n_2$ be their periods.
If $D$ is a disk at $p(\lambda_0)$ such that $p(\lambda_0)$ is the only
periodic point of $f_{\lambda_0}$ of period $\leq n_1$ in $\cl{D}$, 
then it follows from 'dsa' and Rouch\'{e}'s theorem 
that there is a neighbourhood $U(\lambda_0)$ 
of $\lambda_0$ and an integer $n_0\geq 0$ such that for every $\lambda\in U(\lambda_0)$
and every $n\geq n_0$, the map $f_{n,\lambda}$ has exactly one 
periodic point $p_n(\lambda)$ of period $n_1$ in $D$ and no
other periodic point of period $\leq n_1$ in $\cl{D}$.
By the Cauchy Integral Formula (after decreasing the initial disk $D$, if necessary),
every such periodic point $p_n(\lambda)$ must be repelling. 
By the Implicit Function Theorem, every of these points can be 
analytically continued as a repelling periodic point of period $n_1$ 
in a sufficiently small neighbourhood. The previous
observation then implies that for every sufficiently large $n\in\widehat{\N}$, 
there exists an analytic function $p_n: U(\lambda_0)\to \C$ such that 
$p_n(\lambda)\in D$ is a repelling periodic point of $f_{n,\lambda}$ of 
period $n_1$. By construction, $p_n(\lambda)\to p(\lambda)$ when $n\to\infty$.
We can repeat the same procedure for $q(\lambda_0)$ and obtain holomorphic maps 
$q_n: U{'}(\lambda_0)\to\C$. 

Let us now assume that $\lambda_0\in\wt{H}$. 
Since $\wt{H}$ is open, there is a disk $B\subset\wt{H}$ 
centred at $\lambda_0$ such that the maps $p_n, q_n$ are defined 
and holomorphic in $B$ and their images are repelling periodic points
of the corresponding maps (and periods).
Let $U\supset B$ be a simply-connected bounded domain with closure
in $\wt{H}$. Since $\wt{H}$ is a kernel of a sequence $H_n$ of 
nonescaping-hyperbolic components, the compact set
$\cl{U}$ is contained in all $H_n$ for $n\in\N$ chosen sufficiently large.
Hence, the maps $p_n$ and $q_n$ ($n\in\N$) can be holomorphically continued 
to all of $U$, since otherwise the Implicit Function Theorem would
imply that for some $\tilde{\lambda}\in U$, the map $f_{n,\tilde{\lambda}}$
has an indifferent periodic point.

Let us now consider the maps 
$\Phi_{\lambda}(z)=\frac{z-p(\lambda)}{q(\lambda)-p(\lambda)}$ and $\Phi_{n,\lambda}(z)=\frac{z-p_n(\lambda)}{q_n(\lambda)-p_n(\lambda)}$. 
Conjugating $f_{\lambda}$ with $\Phi_{\lambda}$ and 
$f_{n,\lambda}$ with $\Phi_{n,\lambda}$, 
we obtain conformal conjugates such that the points $0$ and $1$ 
correspond to our previously considered repelling periodic points. 
Hence we can assume that $p(\lambda)=0$, $q(\lambda)=1$ for all 
$\lambda\in B$ and $p_n(\lambda)=0$, $q_n(\lambda)=1$ 
for all $\lambda\in U$. In particular, 
\begin{eqnarray*}
\lbrace 0,1\rbrace\subset\J(f_{\lambda}),\J(f_{n,\lambda})\text{ and } S(f_{n,\lambda})\subset\C\backslash\lbrace 0,1\rbrace.
\end{eqnarray*}
for all sufficiently large integers $n\in\widehat{\N}$ 
and the corresponding values of $\lambda$. 

Let $s$ be a singular value of $f_{\lambda_0}$. 
By 'dsa', there is a sequence of singular values $s_n$ 
of the maps $f_{n,\lambda_0}$ that converges to $s$. 
Due to 'hsa', there are holomorphic maps $w_n$ such that 
$w_n(\lambda_0)=s_n$ and $w_n(\lambda)\in S(f_{n,\lambda})$ 
for all $\lambda\in M$. By the previous argument, 
we have that $w_n(U)\subset\C\backslash\lbrace 0,1\rbrace$, 
hence, by Montel's theorem, $\lbrace w_n\rbrace_{n\in\N}$ is a normal 
family on $U$. Let $(w_{n_k})$ be a convergent subsequence of $(w_n)$, 
and let $w$ be a limit function which is necessarily holomorphic. 
By construction we have that $w(\lambda_0)=s$, and 'dsa' implies that 
$w(\lambda)\in S(f_{\lambda})$ holds for all $\lambda\in U$.  
Hence $w$ is a holomorphic parametrization of the singular value $s$ on $U$. 

Consider now for a fixed $\nu$ the family 
$\lbrace f^{\nu}_{n,\lambda}(w_n(\lambda))\rbrace_{n\in\N}$ 
with $\lambda\in U$. Since the Fatou set of an entire map is completely invariant, $w_n(\lambda)\in\F(f_{n,\lambda})$ implies that  $f^{\nu}_{n,\lambda}(w_n(\lambda))
\subset\F(f_{n,\lambda})\subset\C\backslash\lbrace 0,1\rbrace$.  
Applying Montel's theorem it follows that for each $\nu$ the above family is normal in $U$.
For simplicity, denote its limit by $S_{\nu}:U\rightarrow\widehat{\C}$, $\lambda\mapsto S_{\nu}(\lambda):=\lim_{n\rightarrow\infty} f_{n,\lambda} ^{\nu} (w_n(\lambda))$. 
It follows from the local uniform convergence of the maps $f_{n}$ and $w_n$ that 

\begin{eqnarray*}
S_{\nu}(\lambda)=\lim_{n\rightarrow\infty} \lim_{m\rightarrow\infty} f^{\nu}_{n,\lambda} (w_m(\lambda)) = \lim_{n\rightarrow\infty} f^{\nu}_{n,\lambda}(w(\lambda))= f^{\nu}_{\lambda}(w(\lambda)). 
\end{eqnarray*}

By Hurwitz's theorem, either 
$S_{\nu}\subset\C\backslash\lbrace 0,1\rbrace$ or $S_{\nu}\equiv 0$ or $1$. 
Recall that $0$ and $1$ are periodic points of $f_{\lambda}$, 
hence if $S_{\nu}\equiv 0$ (resp. $1$) then there exists some 
$m\in\N$ such that $S_{\nu +km}\equiv 0$ 
(resp. $1$) for all $k\in\N$. 
Applying Montel's theorem once more, 
we obtain that $\lbrace S_{\nu}\rbrace_{\nu\in\N}$ is a normal family on $U$.
\end{proof}

We can now prove the remaining statement which will
then imply Theorem \ref{thm_main}.

\begin{Theorem}
\label{maintheorem}
Let $\widetilde{H}$ be a kernel of a sequence of 
components of $\H(\Fln)$. 
If there is a component $H_{\infty}$ of $\H(\Fli)$ such that 
$\widetilde{H}\cap H_{\infty}\neq\emptyset$, 
then $\widetilde{H}\subset H_{\infty}$.
\end{Theorem}

\begin{proof}
We will prove the statement by contradiction, so assume that there exists some $\lambda_0\in\widetilde{H} \cap\partial H_{\infty}$.
With the same notations as in the previous proof, it follows that $S_{\nu}(B)\subset\C\backslash\lbrace 0,1\rbrace$ for all $\nu\in\N$, 
since $B^{'}:=B\cap H_{\infty}\neq\emptyset$ 
(and hence $S_{\nu}\vert_B$ cannot be constant $0$ or $1$).

For a limit function $S$ of $S_{\nu}$ we have 
that either $S\equiv c\in\widehat{\C}$, in which case $S\equiv\infty$, 
or $S$ is a non-constant function with $S(B)\subset\C\backslash\lbrace 0,1\rbrace$. 

The case $S\equiv\infty$ clearly cannot occur since 
this would mean that some singular value converges to $\infty$ for all
parameters in $B$ but the nonempty subset 
$B'$ of $B$ consists 
of nonescaping-hyperbolic parameters.

Let $S_{\nu}(B)\subset\C\backslash\lbrace 0,1\rbrace$. 
For all $\lambda\in B^{'}$ there is a holomorphic map 
$a:B^{'}\rightarrow\C$ such that $a(\lambda)$ is 
an attracting point of $f_{\lambda}$ which attracts $w(\lambda)$. 
Since $B\supset B^{'}$ is simply-connected, the point $a(\lambda)$ 
can then be analytically continued to an attracting point 
on the entire domain $B$. 
The image $a(B)$ is bounded, hence for every $\lambda\in B$ the singular value
$w(\lambda)$ is attracted by a finite attracting periodic cycle
of $f_{\lambda}$. Since $s=w(\lambda_0)$ was assumed to be 
an arbitrary singular value of $f_{\lambda_0}$, we can repeat this 
procedure for any of the singular values of $f_{\lambda_0}$. 
Recall that, by assumption, 
$\Fli\subset\Hol^{*}_b(\C)$, hence the singular sets are bounded.
This implies that  
$B\ni\lambda_0$ is contained in some 
nonescaping-hyperbolic component $H_{\infty}$ of $\Fli$, 
contradicting the assumption that 
$\lambda_0\in\widetilde{H}\cap\partial H_{\infty}$.
\end{proof}

Without the assumption of dynamical approximation, 
(which is part of 'dsa') 
we cannot expect that a kernel $\widetilde{H}$ is always a nonescaping-hyperbolic 
component of the family $\Fli$. It is easy to find suitable examples. 
One such example was given in \cite{krauskopf_kriete_3}: 
the authors approximated a holomorphic family of quadratic 
polynomials by families of polynomials of degree four, such that
a kernel of a sequence of nonescaping-hyperbolic components 
was a proper subset of some 
nonescaping-hyperbolic component of the limit family.

Here we give an example which respects our standing assumptions, 
showing that the case $\widetilde{H}\cap\H(\Fli)=\emptyset$ 
in Theorem \ref{thm_main} does indeed occur.

\begin{exm}
Let 
\begin{eqnarray}
P_{\lambda,n}(z)=z^3 -\left( \frac{\lambda -1}{\sqrt{\mu_n +\lambda -2}} +\sqrt{\mu_n +\lambda -2}\right)  z^2 +\lambda z,
\end{eqnarray}
where $\lambda\in\C\setminus [1,5]$, $\vert\mu_n\vert<1$ does not 
depend on $\lambda$ and $\mu_n\rightarrow 1$ as $n\rightarrow\infty$ 
(for instance we can choose $\mu_n=\frac{n}{n+1}$). 
Note that $P_{\lambda,n}(z)$ and 
$P_{\lambda}(z)=z^3 - 2\sqrt{\lambda-1} z^2 + \lambda z$ 
satisfy 'dsa' and 'hsa'. 
Every $P_{\lambda,n}$ has $0$ as a fixed point of multiplier $\lambda$. 
Furthermore, there is a second fixed point 
$a_n=\sqrt{\mu_n +\lambda -2}$ with multiplier $\mu_n$. 
Thus, if we choose $\lambda\in\D$ then every polynomial $P_{\lambda,n}$ 
is nonescaping-hyperbolic. 
Hence there is a kernel $\widetilde{H}$ of 
components of $\H(\Fln)$ which contains the unit disk $\D$. 
On the other hand, 
every $P_{\lambda}$ has a parabolic fixed point at 
$a=\sqrt{\lambda -1}$, hence $\H(\Fli)=\emptyset$. 
\end{exm}

Nevertheless, the behaviour of $f_{\lambda}$ is still 
stable in the sense of \emph{$J$-stability} for 
parameters belonging to a kernel. 
Here, $J$-stability is defined analogously to the case of rational maps 
or transcendental entire maps with finitely many singular values 
(see e.g. \cite{eremenko_lyubich_2}):

\begin{Definition}
\label{dfn_Jstable}
Let $\Fl=\lbrace f_{\lambda}:\C\rightarrow\C: \lambda\in M\rbrace$ 
be a holomorphic family of entire functions. 
A map $f_{\lambda_0}\in G$ is said to be \emph{$J$-stable} 
if $\J(f_{\lambda_0})$ moves holomorphically in a neighbourhood 
$\Lambda\subset M$ of $\lambda_0$, i.e., if there is a holomorphic motion 
\begin{eqnarray*}
\Phi:\Lambda\times\J(f_{\lambda_0})\rightarrow\C
\end{eqnarray*}
such that 
\begin{eqnarray*}
\Phi(\lambda,\J(f_{\lambda_0}))=\J(f_{\lambda})\text{ and } 
\Phi(\lambda, f_{\lambda_0}(z))=f_{\lambda}(\Phi(\lambda,z))
\end{eqnarray*}
for all $z\in\J(f_{\lambda_0})$ and all $\lambda\in\Lambda$.
\end{Definition}
$\Phi$ being a holomorphic motion means 
that $\Phi$ is injective in $z$ when $\lambda$ is fixed, holomorphic
in $\lambda$ for fixed $z$ and that
$\Phi_{\lambda_{0}}\equiv\id$. 
For more details see \cite[Section 8]{eremenko_lyubich_2}.

The map $\Phi_{\lambda}$ is a conjugacy between 
$f_{\lambda_0}$ and $f_{\lambda}$ on their Julia sets, hence it 
maps periodic points of $f_{\lambda_0}$ 
to periodic points of $f_{\lambda}$. 
Since periodic points form a dense subset of the Julia set, 
such a conjugacy is unique if it exists.

\begin{Theorem}
Let $\widetilde{H}$ be a kernel of a sequence of 
components of $\H(\Fln)$ and let $\lambda\in\widetilde{H}$. 
Then $f_{\lambda}\in\Fli$ is $J$-stable.
\end{Theorem}

\begin{proof}
We will prove the contraposition, 
so let $\lambda_0$ be a parameter for which 
$f_{\lambda_0}$ is not J-stable 
and let $\Lambda$ be a simply-connected bounded neighbourhood of $\lambda_0$. 
Then there is some repelling periodic point, 
say of period $n$, of $f_{\lambda_0}$ 
which has no analytic continuation as a solution of the equation 
$f_{\lambda}^n(z)-z=0$. 
Otherwise, it would follow from the $\lambda$-lemma \cite{eremenko_lyubich_2} 
that the closure of all repelling points of $f_{\lambda_0}$, 
which equals the Julia set of $f_{\lambda_0}$, 
would move holomorphically, contradicting that $f_{\lambda_0}$ 
is not $J$-stable. 

Let $p(\lambda_0)$ be such a repelling periodic point and let 
$\gamma:\left[ 0,1\right] \rightarrow\Lambda$, $\gamma(0)=\lambda_0$, 
be a path along which $p(\lambda_0)$ cannot be continued analytically. 
By the Implicit Mapping Theorem, the point $p(\lambda)$ 
for $\lambda=\gamma(1)$ must be indifferent. 
By the Minimum Modulus Principle we can then find a nearby 
path $\tilde{\gamma}\subset\Lambda$ connecting $\lambda_0$ to some 
parameter $\tilde{\lambda}$ along which the considered point 
becomes attracting. Hence there is a singular value 
$s(\tilde{\lambda})$ of $f_{\tilde{\lambda}}$ 
converging to the attracting periodic point $p(\tilde{\lambda})$. 

Now let us assume that our assumption is wrong, 
meaning that $\lambda_0$ belongs to some kernel $\widetilde{H}$. 
We can assume w.l.o.g. that the neighbourhood $\Lambda$ was 
chosen sufficiently small such that $\Lambda\Subset\widetilde{H}$.
 Then Theorem \ref{normality} implies that there is a holomorphic 
parametrization $w$ of the singular value $s(\tilde{\lambda})$ such that $\lbrace f^n_{\lambda}(w(\lambda))\rbrace$ is a normal family on $\Lambda$. 
But then each $w(\lambda)$ converges to an attracting point 
$p(\lambda)$ of $f_{\lambda}$, in which case $p(\lambda)$ can 
be continued analytically to an attracting point of $f_{\lambda}$ 
on the whole of $\Lambda$, contradicting the fact that $p(\lambda_0)$ is repelling. 
\end{proof}

\section{Construction of examples}
\label{section_qc}

As we have seen in the previous section, 
a sequence of families to which our results apply has to satisfy
two primary conditions which we have formulated as the standing assumptions
'dsa' and 'hsa'.
Hence starting with an entire map $f\in\Hol^{*}_b(\C)$, we can split the problem 
into the following two:
 
\begin{itemize}
\item[$(i)$] Find a sequence $f_n$ of entire functions which approximates $f$ \emph{dynamically}.
\item[$(ii)$] Construct \emph{holomorphic} families $\lbrace f_{n,\lambda}\rbrace$ and $\lbrace f_{\lambda}\rbrace$ using the functions $f_n, f$. 
\end{itemize}
We will start with the second problem. 
It turns out that there is a natural way to find suitable holomorphic families 
for \emph{any} entire function. 

\subsection*{Holomorphic families inside quasiconformal equivalence classes}
Recall that $M$ is a complex manifold.
Let $\mu$ be a k-Beltrami coefficient of $\C$, i.e., 
$\mu:\C\rightarrow\C$ is a measurable function 
such that $\Vert\mu\Vert_{\infty}\leq k<1$ 
almost everywhere (a.e.) in $\C$. 
By the Integrability Theorem \cite[Theorem 4.4]{lehto}, 
there exists a k-quasiconformal homeomorphism 
$\Psi:\C\rightarrow\C$ whose complex dilatation 
equals $\mu$ a.e. in $\C$. 

Let $f\in\Hol^{*}_b(\C)$ be an entire map. 
Then the \emph{pull-back} $f^{*}\mu$ of $\mu$ by $f$ 
is given by
\begin{eqnarray*}
f^{*}\mu(z):=(\mu\circ f)\frac{\overline{f^{'}(z)}}{f^{'}(z)}.
\end{eqnarray*}
In particular, $\Vert\mu\Vert_{\infty}\leq k$ implies that 
$\Vert f^{*}\mu\Vert_{\infty}\leq k$.

By the Integrability Theorem  
there exists a quasiconformal homeomorphism $\Phi$ 
whose complex dilatation equals $f^{*}\mu$ a.e. in $\C$. 
A formal computation yields that the complex dilatation of the 
map $g=\Psi\circ f\circ\Phi^{-1}$ is $0$ a.e., which means that 
$g$ is a holomorphic map \cite[Theorem 1.1]{lehto}. 
We say that $g$ is \emph{quasiconformally equivalent} to $f$. 

Let $\Lambda\subset M$ be an open connected set and 
let $\lbrace\Psi_{\lambda}\rbrace$ be a family 
of quasiconformal homeomorphisms with uniformly bounded complex dilatations 
that depend holomorphically on $\lambda\in \Lambda$.
The parametrized version of the 
Integrability Theorem \cite[Chapter 4.7]{hubbard}
gives the following way to construct such a family,
starting from a Beltrami coefficient $\mu_0$:
 
Let $(\mu_{\lambda})_{\lambda\in\Lambda}$ be a holomorphic family 
of Beltrami coefficients with $\Vert \mu_{\lambda}\Vert_{\infty}\leq k<1$ 
that contains $\mu_0$, i.e.,
$\mu_0\equiv\mu_{\lambda_0}$ for some $\lambda_0\in\Lambda$.
(One way to embed $\mu_0$ in a holomorphic family
of Beltrami coefficients is as follows: Let
 $h:\Lambda\to D_r(0)$ be 
a holomorphic map, where $r$ is chosen such that 
$ (1+r)\cdot \Vert\mu_0\Vert_{\infty}<1$ and $h(\lambda_0)=0$. 
Then the functions $\mu_{\lambda}:= (1+h(\lambda))\cdot\mu_0$
form a holomorphic family of Beltrami coefficients.)
For every $\lambda$, let $\Psi_{\lambda}$ be a quasiconformal homeomorphism
with Beltrami coefficient $\mu_{\lambda}$, chosen such that
all $\Psi_{\lambda}$ have the same parametrization (e.g., all
$\Psi_{\lambda}$ fix $0$, $1$ and $\infty$). Then the Integrability
Theorem implies that $\lambda\mapsto \Psi_{\lambda}$ is also holomorphic.

As in the previous construction we obtain a family 
$\Fli=\lbrace f_{\lambda}=\Psi_{\lambda}\circ f
\circ\Phi^{-1}_{\lambda}\rbrace_{\lambda\in\Lambda}$ 
of entire maps. By computing the derivative of the equation 
$\Psi_{\lambda}\circ f=f_{\lambda}\circ\Phi_{\lambda}$ 
with respect to $\overline{\lambda}$, we get that the functions 
$f_{\lambda}$ also depend holomorphically on $\lambda$ 
(see e.g. the proof of Proposition 13 in \cite{buff_cheritat}). 
Since an entire function $f$ is a covering map when 
restricted to $\C\backslash f^{-1}(S(f))$, we obtain 
that $S(f_{\lambda})=\Psi_{\lambda}(S(f))$, 
hence the set of singular values moves 
holomorphically in the family $\Fli$ in a unique way. 
Furthermore, $S(f_{\lambda})$ is bounded for all $\lambda$
since the maximal dilatations of the 
quasiconformal maps $\Psi_{\lambda}$ are uniformly bounded.

Let $(f_n)$ be a sequence of entire 
functions which converges to $f$ dynamically. 
In the same manner as above, we define for each $n$ the holomorphic family 
$\Fln=\{ f_{n,\lambda}=\Psi_{\lambda}\circ f_n
\circ\Phi^{-1}_{n,\lambda}\}_{\lambda\in\Lambda}$ 
by starting with the same family $\lbrace\Psi_{\lambda}\rbrace$ 
of quasiconformal homeomorphisms. As before, the set of singular values  $S(f_{n,\lambda})=\Psi_{\lambda} (S(f_n))$ moves holomorphically. 
Hence it remains to show that the map 
\begin{eqnarray*}
F:M^{'}\rightarrow\Hol_b^{*}(\C),\; (n,\lambda)\mapsto f_{n,\lambda}
\end{eqnarray*}
is continuous. In other words, let $\lambda_0\in M$ 
and let $(n,\lambda)\rightarrow (\infty, \lambda_0)$. 
We have to show that in this case, 
\begin{eqnarray*}
\chi_{\luc}(f_{n,\lambda},f_{\lambda_0}) 
+ d_H(S(f_{n,\lambda}),S(f_{\lambda_0}))\rightarrow 0.
\end{eqnarray*}

Clearly, $d_H(S(f_{n,\lambda}),S(f_{\lambda_0}))
=d_H(\Psi_{\lambda}(S(f_n)),\Psi_{\lambda_0}(S(f)))\rightarrow 0$ 
as $(n,\lambda)\rightarrow (\infty, \lambda_0)$, 
since $\Psi_{\lambda}$ depends holomorphically 
(so in particular continuously) 
on $\lambda$ and the maps $f_n$ approximate $f$ dynamically.

To see that $\chi_{\luc}(f_{n,\lambda},f_{\lambda_0})\rightarrow 0$, 
we have to look at the sequence of pull-backs 
\begin{eqnarray*}
f_n^{*}\mu_{\lambda}=(\mu_{\lambda}\circ f_n)\frac{\overline{f_n^{'}(z)}}{f_n^{'}(z)},
\end{eqnarray*}
where $\mu_{\lambda}$ denotes the complex dilatation of $\Psi_{\lambda}$.
By assumption, there exists a constant $k<1$ such that 
$\Vert f_n^{*}\mu_{\lambda}\Vert_{\infty}<k<1$ for all $n$. 
Since also $f_n^{*}\mu_{\lambda}\rightarrow f^{*}\mu_{\lambda_0}$ 
a.e., it follows that 
the (uniquely normalized solutions) $\Phi_{n,\lambda}$ converge locally 
uniformly to $\Phi_{\lambda_0}$ \cite[Theorem 4.6]{lehto}, 
yielding the desired statement. 

\subsection*{Functions of sine type}
The above construction yields many examples of 
families where Theorems \ref{normality} and \ref{maintheorem} 
can be applied to, provided that we have a 
sequence of functions approximating $f$ dynamically. 
There are clearly various ways of approximating 
an entire function locally uniformly but it is a very 
strong requirement to keep control over the sets 
of singular values, since this set can be arbitrarily complicated 
(for instance, it can have nonempty interior). 

For certain (families of) transcendental entire 
functions that were and are of particular interest in 
holomorphic dynamics, appropriate approximations are known and were 
extensively studied. 
For instance, one can approximate the exponential map by the 
polynomials $P_n(z)=(1+z/n)^n$ or the function $f(z)=\sin(z)/z$,
which has infinitely many singular values,
by the sequence $T_n(z)/z$, where $T_n(z)$ denotes the 
Chebyshev polynomial of degree $n$. 

We want to introduce another set of transcendental entire 
functions for which a dynamical approximation exists, 
namely \emph{real sine-type maps with real zeros}.

\begin{Definition}
A function $f$ is said to be of 
\emph{sine type $\sigma$} if there are positive constants $c, C, \tau$ such that 
\begin{eqnarray*}
c\e^{\sigma\vert\Ima z\vert}\leq\vert f(z)\vert\leq C\e^{\sigma\vert\Ima z\vert},
\end{eqnarray*}
where the upper estimate holds everywhere in $\C$ and the lower
estimate holds at least outside the horizontal
strip $\vert\Ima z\vert\leq\tau$.
\end{Definition}

Let $f$ be a sine-type function and denote by $z_n$ the zeros of $f$. 
By \cite[Lecture 17]{levin} or (\cite[Theorem 3]{semmler}), the limit  
\begin{eqnarray*}
\lim_{R\rightarrow\infty}\prod_{\vert z_n\vert\leq R} \left( 1-\frac{z}{z_n}\right) 
\end{eqnarray*}
exists uniformly on compact subsets of $\C$, 
and it defines an entire function called the \emph{generating function of the sequence $(z_n)$} 
which equals $f$ up to $K\cdot z^m$ where $K$ is a constant and 
$m\geq 0$ is an integer \cite[Theorem 2]{semmler}. 
By definition, the zeros of $f$ are contained 
in a horizontal strip around the real axis and  
$f$ has exactly two tracts over $\infty$, 
each of which contains some upper and lower halfplane, respectively. 
It follows from the Ahlfors-Denjoy Theorem 
\cite[XI, \S 4, 269]{nevanlinna} that $f$ has at most 
two asymptotic values. 
The derivative of the generating function $\tilde{f}$ of $(z_n)$ is given by
$\tilde{f}^{'}=\tilde{f}\cdot\sum\frac{1}{z-z_n}$ and 
an elementary computation shows that $\tilde{f}$ and hence $f$ has no
critical points outside a sufficiently wide horizontal strip.
So the set of 
critical values of $f$ is bounded, 
implying that every $f$
belongs to the Eremenko-Lyubich class $\B$.

It is now easy to show the following.
\begin{Proposition}
Let $f$ be a real sine-type function for which all zeros are real.
Then there exists a sequence $p_n$ of polynomials such that
$\chi_{\dyn}(p_n,f)\to 0$ when $n\to\infty$. 
\end{Proposition}

\begin{proof}
By a theorem of Laguerre \cite[Chapter 8.51]{titchmarsh}, 
all zeros of $f^{'}(z)$ are real and are 
separated from each other by the zeros of $f$. 
By basic calculus arguments, the generating polynomials 
cannot have ``free'' critical values, 
and by \cite[Theorem 2]{kisaka} each singular value of 
$f$ is approximated by a sequence of singular values of 
the generating polynomials. 
Hence the sets of singular values converge in the Hausdorff metric. 
\end{proof}

\begin{Remark}
Standard approximation methods in function theory  
do not relate to the sets of singular values in a way that would enable us to construct 
a sequence of entire maps that dynamically approximate a given function $f$ in a 
nontrivial way.  
This leaves open an interesting function-theoretic question.
\end{Remark}


\begin{thebibliography}{99}
%
\bibitem{bergweiler_1}
{W. Bergweiler}, `Iteration of meromorphic functions', {\em
Bull. Amer. Math. Soc. }29, no. 2 (1993), 151--188, arXiv:math.DS/9310226.
%
\bibitem{bergweiler_etal}
{W. Bergweiler, M. Haruta, H. Kriete, H. G. Meier and N. Terglane}, `On the limit functions of iterates in wandering domains', {\em
Ann. Acad. Sci. Fenn., Ser A I Math. }18 (1993), 369--375.
%
\bibitem{buff_cheritat}
{X. Buff and A. Cheritat}, `Upper bound for the size of quadratic Siegel disks', {\em
Inventiones Mathematicae }156, no. 1 (2004), 1--24.
%
\bibitem{devaney_goldberg_hubbard}
{R. L. Devaney, L. R. Goldberg and J. H. Hubbard}, `A dynamical approximation to the exponential map by polynomials', {\em
Preprint, Berkley } (1986).
%
\bibitem{devaney_etal}
{C. Bodel\'{o}n, R. L. Devaney, M. Hayes, G. Roberts, L. R. Goldberg and J. H. Hubbard}, `Dynamical convergence of polynomials to the exponential', {\em
Journal of Difference Equations and Applications }6 (2000), 275--307.
%
\bibitem{eremenko_1}
{A. E. Eremenko}, `On the iteration of entire functions', {\em Dynamical Systems and Ergodic Theory, Proc. Banach Center Publ. Warsaw }23 (1989), 339--345.
%
\bibitem{eremenko_lyubich_2}
{A. E. Eremenko \and M. Y. Lyubich}, `Dynamical properties of some classes of entire functions', {\em Ann. Inst. Fourier, Grenoble }42, no. 4 (1992), 989--1020.
%
\bibitem{golusin}
{G. M. Golusin}, `Geometrische Funktionentheorie', {\em VEB Deutscher Verlag Der Wissenschaften } (1957).
%
\bibitem{goldberg_keen}
{L. R. Goldberg and L. Keen}, `A finiteness theorem for a dynamical class of entire functions', {\em Ergodic Theory Dynamical Systems }6 (1986), 183--192.
%
\bibitem{hubbard}
{J. H. Hubbard}, `Teichm\"{u}ller theory and applications to geometry, topology and dynamics Volume I: Teichm\"{u}ller theory. Vol. 1', {\em Matrix Editions, Ithaca, New York }(2006).
%
\bibitem{kisaka}
{M. Kisaka}, `Local uniform convergence and convergence of Julia sets', {\em Nonlinearity }8 (1995), 273--281.
%
\bibitem{kisaka_shishikura}
{M. Kisaka and M. Shishikura}, `On multiply connected wandering domains of entire functions', in:  Transcendental dynamics and complex analysis (P. J. Rippon and G. M. Stallard, eds.), no. 348, {\em London Mathematical Society Lecture Notes Series, Cambridge University Press }(2008).
%
\bibitem{krauskopf_kriete}
{B. Krauskopf and H. Kriete}, `Kernel convergence of hyperbolic components', {\em Ergodic Theory Dynamical Systems }17 (1997), 1137--1146.
%
\bibitem{krauskopf_kriete_3}
{B. Krauskopf and H. Kriete}, `On the convergence of hyperbolic components in families of finite type', {\em Mathematica Gottingensis - Schriftenreihe des Mathematischen Instituts der Universit\"{a}t G\"{o}ttingen } (1995).
%
\bibitem{lehto}
{O. Lehto}, `Univalent Functions and Teichm\"{u}ller Spaces', {\em Graduate Texts in Mathematics 109, Springer-Verlag  New York, Berlin, Heidelberg } (1986).
%
\bibitem{levin}
{B. Ya. Levin}, `Lectures on entire functions', {\em Translations of Mathematical Monographs, American Mathematical Society }150, no. 150 (1996).
%
\bibitem{mane_sad_sullivan}
{R. Ma{\~n}\'e and P. Sad and D. Sullivan}, `On the dynamics of rational maps', {\em Annales scientifiques de l'\'Ecole Normale Sup\'{e}rieure (S\'{e}r. 4) }16, no. 2 (1983), 193--217.
%
\bibitem{milnor}
{J. Milnor}, `Dynamics in one complex variable', {\em Annals of Mathematics Studies, Princeton University Press, Princeton, NJ, Vol. 160, 3rd edition }(2006).
%
\bibitem{nevanlinna}
{R. Nevanlinna}, `Eindeutige analytische Funktionen', {\em Springer-Verlag/Berlin, G\"{o}ttingen, Heidelberg, 2nd edition }(1953).
%
\bibitem{semmler}
{G. Semmler}, `Complete interpolating sequences, the discrete Muckenhoupt condition, and conformal mapping', {\em Preprint }(2007), arXiv:math.DS/0708.2787.
%
\bibitem{titchmarsh}
{E. C. Titchmarsh}, `The Theory of Functions', {\em Oxford University Press }(1932).
%
\end{thebibliography}
\end{document}